\documentclass[12pt]{amsart}
\usepackage{amsfonts,amssymb,amscd,amsmath,amsthm,graphicx}

\author{Azadeh Nikou and Anthony G. O'Farrell}

\address[A. Nikou]{Department of Mathematics, Tarbiat Moallem University, 599 Taleghani Avenue, 15618 Tehran, Iran}
\address[A. O'Farrell]{Department of Mathematics and Statistics, NUI, May\-n\-ooth, Co. Kildare, Ireland}
\email[A. Nikou]{a.nikou81@gmail.com}
\email[A. O'Farrell]{admin@maths.nuim.ie}



\title[Ditkin conditions]
{Ditkin conditions}
\date{\today}

\newtheorem{thm}{Theorem}
\newtheorem{cor}{Corollary}[section]
\newtheorem{lem}[cor]{Lemma}
\newtheorem{prop}[cor]{Proposition}
\theoremstyle{definition}
\newtheorem{example}[cor]{Example}
\newtheorem{definition}{Definition}[section]
\theoremstyle{remark}
\newtheorem{rem}{Remark}[section]

\newcommand{\C}{\mathbb{C}}
\newcommand{\N}{\mathbb{N}}

\newcommand{\clos}{\textup{clos}}

\newcommand{\coz}{\textup{coz}}
\newcommand{\rad}{\textup{rad}}
\newcommand{\supp}{\textup{supp}}

\begin{document}

\bibliographystyle{amsplain}

\begin{abstract}
This paper is about the connection between certain Banach-algebraic
properties of a commutative Banach algebra $E$ with unit and the associated
commutative Banach algebra $C(X,E)$ of all continuous functions
from a compact Hausdorff space $X$ into $E$.  The properties concern
Ditkin's condition and bounded relative units.
We show that these properties are shared by $E$ and $C(X,E)$.
We also consider the relationship between these properties
in the algebras $E$, $B$ and $\tilde B$ that appear in
so-called admissible quadruples $(X,E,B,\tilde B)$.
\end{abstract}
\maketitle

\section{Introduction and Preliminaries}
\subsection{}
Let $A$ be a  commutative Banach algebra
with unit. The Gelfand transform
$f\mapsto \hat f$ is a unital algebra
homomorphism from $A$ onto
an algebra $\hat A$ 
of continuous complex-valued functions
on its character space $M(A)$, 
the set of nonzero complex-valued multiplicative
linear functionals on $A$, equipped with the
relative weak-star topology from the dual $A^*$.
The kernel of this homomorphism
is the Jacobson radical $\rad(A)$, and
so $\hat A$ is isomorphic to $A$ when
$A$ is semisimple.
See \cite{Dales} for background. 

For
a nonempty compact Hausdorff space $X$ and a Banach algebra $E$,
we let
$C(X,E)$ be the
space of all continuous maps from $X$ into $E$. We define the
{\em uniform norm }on $C(X,E)$ by
$$\Vert f\Vert_{X}={\sup}_{x\in X}\Vert f(x)\Vert,\quad f\in C(X,E). $$
For $f, g \in C(X,E)$ and $\lambda\in \mathbb C$, the pointwise
operations $\lambda f $, $f+g$ and $fg$ in $C(X,E)$ are defined as
usual. It is easy to see that $C(X,E)$, equipped with
the norm $\Vert \cdot \Vert_{X}$ is a
Banach algebra. If $E=\mathbb C$ we get the algebra
$C(X,\mathbb C)=C(X)$ of all continuous complex-valued functions on
$X$.  Hausner \cite{Hausner} showed that if $E$
is a commutative semisimple algebra, then 
$C(X,E)$ is also semisimple, with character space
homeomorphic to $X\times M(E)$.

In this paper, we consider the connection between certain
Banach-algebraic properties of commutative $E$ and of $C(X,E)$.
In many cases, properties of $E$ are inherited by $C(X,E)$.
The properties concerned will be detailed shortly.
We also consider inheritance of properties by
certain subalgebras of $C(X,E)$ called $E$-valued 
function algebras. More specifically, we 
consider $E$-valued function algebras $\tilde B$
that appear in what are called admissible quadruples
$(X,E,B,\tilde B)$.  We now explain this concept. 

\subsection{$E$-valued function algebras}
We recall definitions from our previous paper \cite{NO}:

\begin{definition}
By an {\em $E$-valued function algebra on $X$} we mean
a subalgebra $A\subset C(X,E)$, equipped with
some norm  that makes it complete, such that (1)
$A$ has as an
element the constant function $x\mapsto 1_E$, (2) $A$
separates points on $X$, i.e.
given distinct points $a,b\in X$, there exists
$f\in A$ such that $f(a)\not=f(b)$, and
(3) the evaluation map
$$
e_x: \left\{
\begin{array}{rcl}
A&\to& E,\\
f&\mapsto&f(x)
\end{array}
\right.
$$
is continuous, for each $x\in X$.
\end{definition}

\begin{definition}
By an {\em admissible quadruple} we mean a quadruple
$(X,E,B,\tilde{B})$, where
\begin{enumerate}
\item\label{4p1}
$X$ is a compact Hausdorff space,
\item \label{4p2}$E$ is a commutative Banach algebra with unit,
\item \label{4p3} $B\subset C(X)$ is a natural $\C$-valued function
algebra on $X$,
\item \label{4p4} $\tilde{B}\subset C(X,E)$ is an $E$-valued function
algebra on $X$,
\item \label{4p5}  $B\cdot E \subset \tilde{B}$, and
\item \label{4p6} \label{last}$\{\lambda\circ f, f\in \tilde{B}, \lambda\in M(E)\}\subset B.$
\end{enumerate}
\end{definition}

One example is $(X,E,C(X),C(X,E))$. For other examples, such as 
Lipschitz algebras and algebras associated to $E$-valued
polynomials, rational functions and analytic functions,
see \cite{NO},
and see also Subsection \ref{SS:Tomiyama}.

Given an admissible quadruple $(X,E,B,\tilde{B})$,
we define the {\em associated map} (also
called Hausner's map)
$$\beta:\left\{
\begin{array}{rcl}
X\times M(E) &\to& M(\tilde{B})\\
(x,\psi) &\mapsto& \psi\circ e_x.
\end{array}
\right.
$$
The associated map is always injective.
\begin{definition}
We say that an admissible quadruple $(X,E,B,\tilde{B})$
is {\em natural} if the associated map $\beta$
is bijective.
\end{definition}

Each quadruple of the form $(X,E,C(X),C(X,E))$ is
admissible and natural. This is a more precise statement
of Hausner's lemma \cite[Lemma 2]{Hausner}.

\subsection{Properties} Let $A$ be a commutative Banach algebra
with unit.

Given an element  $a\in A$, 
the {\em cozero set }of $a$ is defined as
$$\coz(a):=\left\{\phi\in M(A) : \hat a(\phi)\neq 0\right\},$$
and the {\em support }$\supp(a)$
as the closure of $\coz(a)$ in $M(A)$.

To a closed set $S\subset M(A)$ are associated two ideals, 
the {\em kernel }of $S$,
$$I_S= I_S(A):=\{a\in A : \hat{a}(S)\subset\{0\}\},$$
and the smaller ideal
$$J_S=J_S(A):=\{a\in A: \supp (\hat{a})\cap S=\emptyset\}.$$
For $\phi\in M(A)$, we abbreviate $I_{\phi}=I_\phi(A):=I_{\{\phi\}}$ 
(a maximal ideal) and
$J_{\phi}=J_\phi(A):=J_{\{\phi\}}$.

$A$ is said to have {\em bounded relative units }if, for every
$\phi\in M(A)$, there exists $m_{\phi}>0$ such that, for
each compact subset $K$ of $M(A)\setminus\{\phi\}$, there exists
$a\in J_{\phi}$ with $\hat{a}(K)\subset\{1\}$ and $\|a\|\leq
m_{\phi}$.

$A$ satisfies {\em Ditkin's condition at $\phi\in M(A)$}
if $a\in \clos({aJ_\phi})$ for all $a\in I_{\phi}$.  \\
$A$ is a {\em Ditkin algebra }if $A$ satisfies Ditkin's
condition at each $\phi\in M(A)$.

$A$ is a {\em strong Ditkin algebra }if $I_{\phi}$ has a
bounded approximate identity contained in $J_{\phi}$ for each
$\phi\in M(A)$, i.e. there exists $m_\phi>0$
and $u_n\in J_\phi$ ($n\in\N$) such that $\|u_n\|\le m_\phi$
for each $n$ and 
$\|a-au_n\|\to0$ for each $a\in I_\phi$.

\subsection{Summary of Results}

\begin{thm}\label{T:Ditkin}
 Let $X$ be a nonempty compact Hausdorff
space and $E$ be a commutative Banach algebra with unit.
Then $C(X,E)$ is Ditkin if and only if $E$
is Ditkin.
\end{thm}

\begin{thm}\label{T:BRU}\label{T:BRU-referee}
Let
$(X,E,B,\tilde B)$ be a natural admissible quadruple and
suppose $\tilde B$ is semisimple. Then $\tilde B$
has bounded relative units if and only if both
$E$ and $B$ have bounded relative units. 
\end{thm}

\begin{cor}\label{C:BRU}
 Let $X$ be a nonempty compact Hausdorff
space and $E$ be a commutative Banach algebra with unit.
Then $C(X,E)$ has bounded relative units
if and only if $E$ has bounded relative units
\end{cor}

\begin{cor}\label{C:S-Ditkin}
 Let $X$ be a nonempty compact Hausdorff
space and $E$ be a commutative Banach algebra with unit.
Then $C(X,E)$ is a strong Ditkin algebra if and only if $E$
is a strong Ditkin algebra.
\end{cor}

The \lq\lq only if" direction of
Theorem \ref{T:Ditkin} and hence of
Corollary \ref{C:S-Ditkin}
extends to natural admissible quadruples:
(see 
Propositions  
\ref{P:4Dit-only} 
and Corollary \ref{C:4SDit-only}),
but it appears to be unknown whether 
the \lq\lq if" direction does.

The results about quadruples apply to
some so-called
Tomiyama products, defined below.  
See 
Corollaries 
\ref{C:4Dit-Tom},
\ref{C:BRU-Tomiyama}
and \ref{C:TSD}

We conclude the paper with an application to automatic continuity
for maps $T:C(X,E)\to C(Y,F)$. See Section \ref{S:application}.

\subsection{Properties of Admissible Quadruples}
If $(X,E,B,\tilde B)$ is a natural admissible quadruple,
then it is easy to see that
$\tilde B$ is semisimple if and only if $E$ is semisimple.

Although $E$ is not assumed semisimple in the definition,
the quadruple concept really concerns semisimple $E$.  The following
is rather easily checked: 
\begin{prop}\label{P:4hat}
 Let $(X,E,B,\tilde B)$
satisfy conditions (1)-(5) of the
definition.  
Define 
$$ \hat{\tilde B}:= \{ x\mapsto \widehat{f(x)}: f\in \tilde B\}.$$
Then
$(X,E,B,\tilde B)$ is an admissible quadruple,
if and only if $(X,\hat E,B,\hat{\tilde B})$
is an admissible quadruple. \qed
\end{prop}

(We emphasize that, in this proposition, $\hat E$
denotes the Gelfand transform algebra with
the quotient norm from $E/\rad(E)$, not
the supremum norm.)

Also, for semisimple $E$, there is 
sometimes symmetry in the r\^oles of $E$ and $B$:

\begin{definition}
We say that an admissible quadruple $(X,E,B,\tilde B)$ is 
{\em tight} if for each $f\in \tilde B$ the map
\begin{equation}\label{E:Phi}
\Phi(f):
\left\{
\begin{array}{rcl}
M(E) &\to& B \\
\lambda &\mapsto& \lambda\circ f
\end{array}
\right.
\end{equation}
is continuous from $M(E)$ (with the usual 
relative weak-star topology
from $E^*$) to $B$.
\end{definition}

\begin{prop}\label{P:4switch}
Suppose $(X,E,B,\tilde B)$ is a tight admissible
quadruple, and $E$ is semisimple.  Define
$\Phi(f)$ by Equation \eqref{E:Phi}, for each $f\in\tilde B$.
Then $\Phi$ is an algebra isomorphism
of $\tilde B$ onto a $B$-valued function
algebra on $M(E)$, and
$(M(E),B, E,\Phi(\tilde B))$
is an admissible quadruple.
\end{prop}
\begin{proof} Since the quadruple is tight,
the map $\Phi$ 
is a well-defined linear map from the Banach
space $\tilde B$ to the Banach space
$C(M(E),B)$.
An application of the Closed
Graph theorem \cite{Conway} shows that 
$\Phi$
is continuous.  Thus $\Phi(\tilde B)$
is a $B$-valued function algebra on $M(E)$.
The rest is clear.
\end{proof} 

The following example shows that Proposition
\ref{P:4switch} would fail without the assumption
of tightness:

\begin{example}\label{E:not-tight}
Let $C$ denote the set of all continuous functions $f:[0,1]\times[0,2]\to\C$
such that the partial derivative $\displaystyle\frac{\partial f}{\partial y}$
exists at all points of the rectangle $R:= [0,1]\times[0,2]$, is bounded
on the whole rectangle, and is such that $\displaystyle\frac{\partial f}{\partial y}(x,y)$
is continuous on each vertical line, i.e. is 
continuous in $y$ on $[0,2]$ for each fixed $x\in[0,1]$.  With pointwise
operations and the norm
$$ \|f\|_C:= \sup_R|f| + \sup_R\left|\frac{\partial f}{\partial y}\right|,$$
 $C$ is a natural function algebra on $R$.  

Next, take $E= C^0([0,1])$, 
$B=C^1([0,2])$, and $X=[0,1]$.  Then (with pointwise operations
and the usual norms)
$E$ and $B$ are semisimple separable commutative Banach algebras,
with $M(E)=[0,1]$ and $M(B)=[0,2]$.
 
Let
$$ C_1:= \left\{ F\in E^{[0,2]}: 
((x,y)\mapsto F(y)(x)) \in C
\right\}, $$ 
$$ C_2:= \left\{ G\in B^{[0,1]}: 
((x,y)\mapsto G(x)(y)) \in C
\right\}. $$ 
Then $C_1$ is an algebra of $E$-valued functions on $[0,2]$
and $C_2$ is an algebra of $B$-valued functions on $[0,1]$,
when endowed with pointwise operations. Both algebras are
algebra-isomorphic to $C$, via obvious isomorphisms. When they
are given the 
norms induced by these isomorphisms, $(X,E,B,C_1)$ is an
admissible quadruple, and, in the notation of Proposition \ref{P:4switch},
$C_2=\Phi(C_1)$.  

We claim that $(M(E),B,E,\Phi(C_1))=([0,1],B,E,C_2)$ is
not an admissible quadruple, because the elements of
$C_2$ are not all {\em continuous } $B$-valued functions.
To see this, we give an example of a function $f\in C$
such that 
$$\left\{
\begin{array}{rcl}
[0,1] &\to& C^1([0,2]),\\
x&\mapsto& (y\mapsto f(x,y)),\\
\end{array}
\right. 
$$
is not continuous.

Take
$$ f(x,y) = \left\{
\begin{array}{rcl}
0 &,& 0\le y\le x,\\
\displaystyle\frac{(y-x)^2}{2x} &,& x<y <2x,\\
\displaystyle y - \frac{3x}2 &,& 2x\le y\le 2,
\end{array}
\right\},
0\le x\le 1.
$$
Then $f$ is continuous on $R$, the partial derivative
$\displaystyle\frac{\partial f}{\partial y}$ is continuous on each vertical
line and is bounded, but it is not continuous at $(0,0)$.
Moreover, the value of $\|f(x,\cdot)-f(0,\cdot)\|_B$ exceeds
$1$ for all $x>0$, so it does not tend to $0$ as $x\downarrow0$.
\end{example}

\subsection{Tensor products}\label{SS:Tomiyama}
Let $A$ and $B$ be commutative
Banach algebras with unit. A {\em Tomiyama product
of $A$ and $B$ }is the completion of the
algebraic tensor product $A\otimes B$
with respect to some submultiplicative
cross norm not less than the injective tensor
product norm. See \cite{Ryan} and \cite[Section 2.11]{Kaniuth}
for background on cross norms and tensor products
of Banach algebras.    

\begin{prop}\label{P:4Tom}
Let $C$ be a Tomiyama product of $A$ and $B$, two 
commutative Banach algebras with unit. Suppose $C$ is semisimple.
Then $(M(B),A,B,C)$ is a natural
admissible quadruple.
\end{prop}

\begin{rem} Kaniuth shows (cf. \cite[Cor. 2.11.3]{Kaniuth}
that if $C$ is semisimple, then so are $A$ and $B$.
Thus, since we are mainly interested in
semisimple $C$, we might just as well have restricted
to semisimple $A$ and $B$ in the definition of
Tomiyama product. 
 
Tomiyama showed \cite[Theorem 4]{Tom} 
that a Tomiyama product $C$ is automatically semisimple
if both $A$ and $B$ are semisimple, at least 
one of them has the Banach approximation
property and the norm is either the projective
or injective product norm. 
\end{rem}

\begin{proof}[Proof of Proposition]
Let $X=M(B)$. 

First, we have to explain how
$C$ may be regarded as an $A$-valued function
algebra on $X$ (condition (4) of the definition
of admissible quadruple).  

By the definition of Tomiyama product, we have 
$$ \|f\|_{A\check\otimes B} \le \|f\|_C $$
for all $f\in A\otimes B$.

Let $f\in C$. Then there is a $C$-norm-Cauchy
sequence $f_n\in A\otimes B$ with
$\|f-f_n\|_C\to 0$.  Thus
$(f_n)$ is $A\check\otimes B$-norm-Cauchy as well,
and so converges to an element $\Psi(f)\in A\check\otimes B$.
We have 
$$ \|\Psi(f)\|_{A\check\otimes B} =\lim
\|\Psi(f_n)\|_{A\check\otimes B} \le\lim
\|f_n\|_C =
\|f\|_C.$$

One can check that $\Psi(f)$ does not depend on which Cauchy sequence
$(f_n)$ is chosen.  So we have a well-defined continuous
map $\Psi:C\to A\check\otimes B$, a contraction,
in fact.  The map
$\Psi$ is also linear, as is easily seen.

Next, we claim that $\Psi$ is injective. 
Suppose $f\in C$ and $\Psi(f)=0$.  Take any sequence
$f_n\in A\otimes B$ such that $f_n\to f$ in $C$-norm.
Then $f_n\to 0$ in $A\check\otimes B$-norm.  

Fix any $\chi\in M(C)$. By Tomiyama's Theorem,
there exist $\lambda\in M(A)$ and $\gamma\in M(B)$
such that $\chi=\lambda\otimes\gamma$ when restricted to
the algebraic tensor product $A\otimes B\subset C$.
Moreover, there is a character $\chi'$ on 
$A\check\otimes B$ that agrees with $\chi$
on $A\otimes B$.

Fix $\epsilon>0$.  Choose $n\in\N$ such that
$$ \| f- f_n \|_C < \frac{\epsilon}2
\hbox{ and } \| f_n\|_{A\check\otimes B} < \frac\epsilon2.$$
Then
$$ |\chi(f)| \le |\chi(f-f_n)| + | \chi(f_n)| < 
\frac\epsilon2
+\frac\epsilon2,
$$
because $\chi$ has norm $1$ in $C^*$,
$\chi(f_n)=\chi'(f_n)$,  and $\chi'$ has norm $1$
in   $(A\otimes B)^*$.
Thus $|\chi(f)|<\epsilon$ for all $\epsilon>0$.

Thus $\chi(f)=0$ for all $\chi\in M(C)$.
Since $C$ is semisimple, $f=0$.
Thus $\Psi$ is injective, as claimed.

So we have a continuous
injection from $C$ into the injective
tensor product $A\check\otimes B$, which
is a subset of $A\check\otimes C(X)$,
and the latter is naturally identified with
$C(X,A)$, as shown by Grothendieck \cite{Ryan}.

Conditions (\ref{4p1})-(\ref{4p3}) and (\ref{4p5}) are straightforward,
and condition (\ref{4p6}) holds because
$A\otimes B$ is dense in $C$.

Thus $(X,A,B,C)$ is an admissible quadruple.
It is natural by Tomiy\-ama's main result
that $M(C)$ is homeomorphic to $M(A)\times M(B)$
\cite[Theorem 1]{Tom}.
\end{proof}

\begin{rem} The projective tensor product $A\hat\otimes B$ 
of commutative Banach algebras $A$ and $B$ is an example
of a Tomiyama product, but it is not always semisimple
if $A$ and $B$ are.
The natural map $\Psi: A\hat\otimes B\to A\check\otimes B$
(as in the proof above) may fail to be injective.  In fact
\cite[Section 9]{Bierstedt} it is injective for all
$B$ if and only if $A$ has the Banach approximation property.
\end{rem}

\begin{cor}
Let $C$ be a Tomiyama product of $A$ and $B$, two semisimple 
commutative Banach algebras with unit. Suppose $C$ is semisimple.
Then the admissible quadruple $(M(B),A,B,C)$ is tight.
\end{cor}
\begin{proof}
Applying the Theorem with $A$ and $B$ interchanged, we conclude that
$(M(A),B,A,C)$ is an admissible quadruple, so $C$ is
(isometrically isomorphic to) a $B$-valued function algebra
on $M(A)$, i.e. $(M(B),A,B,C)$ is tight.
\end{proof}

Thus we can assert that the algebra $C$ in Example \ref{E:not-tight}
is {\em not} a Tomiyama product of $C^0([0.1])$ and $C^1([0,2])$.

\section{Ditkin algebras}
In this section we will prove Theorem \ref{T:Ditkin}.
As indicated, one direction generalises to
admissible quadruples:

\subsection{The ``only if'' direction}
\begin{prop}\label{P:4Dit-only}
 Let $(X,E,B,\tilde B)$
be an admissible quadruple.
Suppose $\tilde B$ is Ditkin.
Then $E$ and $B$ 
are Ditkin.
\end{prop}
\begin{proof}
Suppose $\tilde B$ is Ditkin.

By Proposition \ref{P:4hat}, $(X,\hat E, B, \hat{\tilde B})$
is admissible, and since $\hat{\tilde B}$ inherits the Ditkin property,
we may assume without loss in generality that $E$ is semisimple.

To see that $E$ is Ditkin, fix $\psi\in M(E)$,
and $b\in E$ with $\psi(b)=0$.  We wish to show that
there exist $b_n\in J_\psi(E)$ such that $\|b-b_nb\|_E\to 0$
as $n\uparrow\infty$.

Pick any $x_0\in X$, and define $\phi=\beta(x_0,\psi)$.
Then $\phi\in M(\tilde B)$. 
Define $f(x)=b$, for all $x\in X$.  Then
$f\in I_\phi$, so since $\tilde B$ is Ditkin
we may choose $f_n\in \tilde B$ such that
each $\hat f_n=0$ near $\phi$ in $M(\tilde B)$
and $\|f-f_nf\|_X\to0$ as $n\uparrow\infty$.
Take $b_n=f_n(x_0)$. Then 
$\|b-b_nb\|_E\to0$.
Since $\beta$ is
continuous, we may choose open sets $U_n\subset X$ and
$V_n\subset M(E)$ such that
$x_0\in U_n$, $\psi\in V_n$ and $\hat f_n=0$
on $\beta(U_n\times V_n)$. Then for $\chi\in V_n$ we have
$$ \hat b_n(\chi) = \hat f_n(\beta(x_0,\chi)) =0.$$
Thus $b_n\in J_{\psi}$, as desired.

If $(X,E, B, \tilde B)$ were tight, we could immediately
use Proposition \ref{P:4switch}, to deduce that
$(M(E),B,E,\Phi(\tilde B))$
is admissible, and the isomorphic algebra $\Phi(\tilde B)$
is Ditkin, so $B$ is Ditkin. However we do not need to make this assumption.

Assume just that 
$(X,E,B,\tilde B)$ is an admissible quadruple,
and $\tilde B$ is Ditkin. 
The map 
$$ \Psi(\lambda):\left\{
\begin{array}{rcl}
\tilde B &\to& B \\
f &\mapsto & \lambda\circ f
\end{array}
\right.
$$
is a well-defined algebra homomorphism, for each $\lambda\in M(E)$.  
By using the Closed Graph theorem,
we see that $\Psi(\lambda)$ is continuous.

Now fix $a\in X$ and $g\in B$ with $g(a)=0$.  Pick any $\lambda_0\in M(E)$
and define $\phi = \beta(a,\lambda_0)$.  Then $\phi\in M(\tilde B)$.
Define $f(x) = g(x) \cdot 1_E$ for all $x\in X$.  Then $f\in \tilde B$ and
$\phi(f)=0$, so since $\tilde B$ is Ditkin we may choose $f_n\in\tilde B$
such that $\hat f_n=0$ near $\phi$ in $M(\tilde B)$ and
$\| f - f_n f\|_{\tilde B}\to0$.  Let $g_n = \Psi(\lambda_0)(f_n)$.
Then $g_n\in B$ and $g_n=0$ near $a$. 
Since $g= \Psi(\lambda_0)(f)$ and $\Psi(\lambda_0)$ is continuous, we have
 $$ \| g - g_n g \|_B \to 0. $$
Thus $B$ is Ditkin.
\end{proof}

Applying Proposition \ref{P:4Tom}, we have:
\begin{cor}\label{C:4Dit-Tom}
Let $A$ and $B$ be semisimple commutative Banach algebras with unit,
and let $C$ be a semisimple Tomiyama product of $A$ and $B$.
Suppose $C$ is Ditkin. Then  
so are $A$ and $B$. \qed 
\end{cor}

\subsection{Converse direction}
Turning to the other direction, 
we restrict to the special
quadruple $(X,E,C(X),C(X,E))$:

\begin{prop}\label{P:Dit-if} Let $X$ be a compact Hausdorff space
and let $E$ be a commutative Banach algebra with unit.
Suppose $E$ is Ditkin.
Then $C(X,E)$ is Ditkin.
\end{prop}
\begin{proof}
Fix $\phi\in M(C(X,E))$,
and $f\in I_\phi$.
Let $\epsilon>0$ be given.  

Choose $\psi\in M(E)$
and $x_0\in X$ such that $\phi=\beta(x_0,\psi)$.

Let $a=f(x_0)$. Then $\psi(a)=0$.
Since $E$ is Ditkin, we may choose
$b\in J_\psi(E)$ such that
$\|a-ba\|_E<\epsilon$. 
Let
$$ U= \{x\in X: \|f(x)b - f(x)\|_E<\epsilon\}.$$
Then $U$ is  an open neighbourhood of $x_0$.
Thus, by Urysohn's Lemma, we may choose $h\in J_{x_0}(C(X))$
with $h=1$ off $U$ and $0\le h\le 1$ on $X$.

Then for each $x\in X$, we have
$$ \|(1-h(x))(bf(x)-f(x))\|_E < \epsilon.$$
Thus $\|(1-h)(bf-f)\|_X<\epsilon$.  Now
$$ f + (1-h)(bf-f) = f(h\cdot1_E+b-hb) \in f J_\phi$$
(since $h\cdot1_E+b-hb=0\ $ on $\beta^{-1}\left(h^{-1}(0)\cap b^{-1}(0)\right)$),
so the distance from $f$ to $fJ_\phi$
in $C(X,E)$ norm is less than $\epsilon$.

The result follows.
\end{proof}

\begin{proof}[Proof of Theorem \ref{T:Ditkin}]
Apply Proposition \ref{P:4Dit-only}
(with $B=C(X)$ and $\tilde B=C(X,E)$)
and Proposition \ref{P:Dit-if}.
\end{proof}

\vspace{1cm}
\section{Bounded relative units}
\subsection{Proof of Theorem \ref{T:BRU}}
Note the following:
\begin{lem}\label{L:referee}
If $A$ is a commutative Banach algebra with identity, then the
following are equivalent:
\\
(1) $A$ has bounded relative units.
\\
(2) For each $\phi\in M(A)$, there exists a constant $c_\phi>0$
with the following property: for every closed subset $K$ of $M(A)$
with $\phi\not\in K$, there exists $x\in A$ such that $\|x\|\le c_\phi$,
$\hat x=0$ on $K$ and $\hat x=1$
on some neighbourhood of $\phi$.
\end{lem}
\begin{proof}
Let $e$ denote the identity of $A$. 

Suppose (1) holds.
Fix $\phi\in M(A)$, and let $m_\phi>0$ be chosen as in the definition
of bounded relative units. Take $c_\phi=m_\phi+\|e\|$.
Let $K\subset M(A)$ be compact, with $\phi\not\in K$. 
We may choose $a\in J_\phi$ such that $\hat a(K)\subset\{1\}$
and $\|a\|\le m_\phi$. Taking $x=e-a$ we have
$\|x\|\le c_\phi$, $\hat x=0$ on $K$ and $\hat x=1$ near $\phi$.
Thus (2) holds.

The other direction is similar.
\end{proof}

This shows, in particular, that a unital commutative
Banach algebra with bounded relative units is regular.

\begin{proof}[Proof of Theorem \ref{T:BRU}]
For the \lq\lq only if" direction, 
suppose $(X,E,B,\tilde B)$ is
an admissible quadruple,
and 
$\tilde B$ has bounded relative units.
Then we have to show that
$E$ and $B$ have bounded relative units.

First, consider $E$, and  fix $\psi_0\in M(E)$. 
Fix any $x_0\in X$. Since
the evaluation map $f\mapsto f(x_0)$ is continuous from 
$\tilde B\to E$, there exists $\kappa>0$ such that
$\|f(x_0)\|_E\le\kappa\|f\|_{\tilde B}$ for all
$f\in \tilde B$. 

Define $\phi:=\beta(x_0, \psi_0)\in M(\tilde B)$.
By assumption, 
there exists $m>0$ 
such that for each open neighbourhood
$W$ of $\phi$ there exists 
$f\in J_{\phi}$
such that $\hat{f}=1$ off $W$ and $\|f\|_{\tilde B}\leq
m$. 
Let $F\subset M(E)\setminus\{\psi_0\}$ be a compact subset.
Define $L:=\{\beta(x_0,\chi) :
\chi\in F\}$. It is clear that $L$ is a compact
subset of $
M(C(X,E))\backslash\{\phi\}$.
We may choose $f\in J_{\phi}$
such that $\hat{f}(L)\subset\{1\}$ and $\|f\|_{\tilde B}\leq
m$. 
Define $b:=f(x_0)$. Then $b\in E$, $\hat{b}(F)\subset\{1\}$ and
$\|b\|_E\leq\kappa\|f\|_{\tilde B}\leq \kappa m$.

Thus $E$ has bounded relative units.

Now consider $B$, and fix $x_0\in X=M(B)$. Fix any $\psi_0\in M(E)$.
As noted in the proof of
Proposition \ref{P:4Dit-only}, the map $f\mapsto \psi_0\circ f$
is continuous from $\tilde B\to B$, so there exists $\kappa>0$
such that $\|\psi_0\circ f\|_B\le\kappa\|f\|_{\tilde B}$.
So defining $\phi:=
\beta(x_0, \psi_0)\in M(\tilde B)$, we may proceed in
a very similar way to the above, to deduce that $B$
has bounded relative units.

\smallskip
For the \lq if' direction, the key observation 
(for which the authors would like to thank the referee)
uses
the classical automatic continuity theorem of
Shilov \cite[Theorem 2.3.3, p. 192]{Dales}
that each homomorphism from a Banach algebra
into a semisimple commutative Banach algebra
is necessarily continuous. We may apply this
to the two homomorphisms 
$$
 \left\{
\begin{array}{rcl}
E &\to& \tilde B\\
a &\mapsto& 1_X\cdot a
\end{array}
\right\}
\textup{ and }
 \left\{
\begin{array}{rcl}
B &\to& \tilde B\\
f &\mapsto& f\cdot e
\end{array}
\right\}
$$
where $e$ is the identity of $E$, and deduce that there
exist constants $\alpha>0$ and $\gamma>0$ such that
$\|1_X\cdot a\|_{\tilde B}\le\alpha\|a\|_E$ for all $a\in E$ and
$\|f\cdot e\|_{\tilde B}\le\gamma\|f\|_B$ for all $f\in B$.

Now every $\phi\in M(\tilde B)$ is of the form $\psi\circ e_x$
for some $\psi\in M(E)$ and some $x\in X$. Let
$c_x$ and $c_\psi$ be constants as guaranteed by the assumption that
$B$ and $E$ have bounded relative units. Since $M(\tilde B)$
is homeomorphic to $X\times M(E)=M(B)\times M(E)$, given a closed
subset $K$ of $M(\tilde B)$ such that $\phi\not\in K$, 
we may find closed subsets $C\subset X$ and $D\subset M(E)$
such that $K\subset (C\times M(E)) \cup (X\times D)$ and
$x\not\in C$ and $\psi\not\in D$. Then, by hypothesis, there exist
\begin{itemize}
\item $f\in B$ such that $\|f\|_B\le c_x$, $f=0$ on $C$ and
$f=1$ on a neighbourhood of $x$;
\item $a\in E$ such that $\|a\|_E\le c_\psi$, $\hat a=1$ on $D$
and $\hat a=1$ on a neighbourhood of $\psi$.
\end{itemize}
Then the element $f\times a$ of $\tilde B$ satisfies
$\widehat{f\cdot a}=0$ on $(C\times M(E)) \cup (X\times D)$
and 
$\widehat{f\cdot a}=1$ in a neighbourhood of $\phi$. Moreover,
$$ \|f\cdot a\|_{\tilde B} \le \|f\cdot e\|_{\tilde B}
\cdot \|1_X\cdot a\|_{\tilde B} \le \alpha\gamma\|a\|_E\|f\|_B
\le \alpha\gamma c_x c_\phi.$$
Thus $\tilde B$ has bounded relative units. 
\end{proof}

\begin{cor}\label{C:BRU-Tomiyama}
Let $A$ and $B$ be semisimple commutative Banach 
algebras with unit.
Suppose that $C$ is a semisimple Tomiyama product
of $A$ and $B$.
Then 
$C$ has bounded relative units if and only if
both
$A$ and $B$ have bounded relative units. \qed
\end{cor}

\subsection{Proof of Corollaries}
\begin{proof}[Proof of Corollary \ref{C:BRU}]
This is immediate from Theorem \ref{T:BRU}, because
$E$ has bounded relative units if and only if $\hat E$ does, 
and $C(X,\hat E)$ is semisimple, so the 
theorem applies to the quadruple 
\\$(X,\hat E,C(X),C(X,\hat E)$,
and tells us that $C(X,\hat E)$
has bounded relative units if and only if $E$ does.  
But 
$C(X, \hat E)$ is isometrically algebra-isomorphic 
to $\widehat{C(X,E)}$, 
so $C(X,E)$ has bounded relative units
if and only if $E$ does.
\end{proof}

\begin{proof}[Proof of Corollary \ref{C:S-Ditkin}]
This follows from Theorems \ref{T:Ditkin} and
\ref{T:BRU}, since an algebra is a strong
Ditkin algebra if and only if it is Ditkin and has
bounded relative units \cite[pp.417-8]{Dales}.
\end{proof}

As indicated earlier, one direction of Corollary
\ref{C:S-Ditkin} generalises to
natural admissible  quadruples:
\begin{cor}\label{C:4SDit-only}
 Let $(X,E,B,\tilde B)$
be an admissible quadruple.
Suppose $\tilde B$ is strong Ditkin.
Then $E$ and $B$ 
are strong Ditkin.
\end{cor}
\begin{proof}
This follows from Proposition \ref{P:4Dit-only} and
Theorem \ref{T:BRU}. 
\end{proof}

\begin{cor}\label{C:TSD}
Let $A$ and $B$ be semisimple commutative Banach algebras with unit,
and let $C$ be a semisimple Tomiyama product of $A$ and $B$.
Suppose $C$ is strong Ditkin. Then  
so are $A$ and $B$. \qed
\end{cor}

\section{Separating bijections}\label{S:application}
\begin{definition}
Let $A$ and $B$ be two semisimple commutative
Banach algebras with identity. A linear map $T:A\rightarrow
B$ is said to be {\em separating } or {\em disjointness preserving }if
$\coz(Tf)\cap \coz(Tg)=\emptyset$ whenever $f,g\in A$ satisfy
$\coz(f)\cap \coz(g)=\emptyset$. Moreover, $T$ is said to be
{\em biseparating }if it is bijective and both $T$ and $T^{-1}$ are
separating.\\
\end{definition}
Equivalently, a map $T:A\rightarrow B$ is separating if it is linear
and $Tf\cdot Tg\equiv 0$, whenever $f,g\in A$ satisfy $f\cdot
g\equiv 0$.
As an application of Theorem \ref{T:Ditkin}, we obtain: 
\begin{thm}\label{}
Let $X, Y$ be two compact Hausdorff spaces and $E, F$ be unital
commutative semisimple Banach algebras which are 
Ditkin algebras and $ T:C (X,E)\rightarrow
C(Y,F)$ be a separating linear bijection, then
\begin{itemize}
\item[$(i)$] $T$ is continuous,
\item[$(ii)$] $T^{-1}$ is separating, and
\item[$(iii)$] $X\times M(E)$ and $Y\times M(F)$ are homeomorphic.
\end{itemize}
\end{thm}
\begin{proof}
Use \cite[Theorem 1]{Juan J.Font} and Theorem \ref{T:Ditkin}.

\end{proof}

\begin{rem} The results of this paper may be extended to
semisimple commutative Banach algebras without identity
by the device of adjoining a unit.  We have confined attention
to algebras with unit, to avoid clutter.
\end{rem}

\subsection*{Acknowledgments}
This paper was begun while the first-named author was visiting
National University of Ireland, Maynooth; she would like to
thank sincerely 
members of 
the Department of Mathematics and Statistics 
for their hospitality and kindness.  The authors are grateful to the
anonymous referee for suggestions that significantly improved the
paper.  In particular, the referee supplied a key part of
the proof of Theorem \ref{T:BRU} and the present proof of
Corollary \ref{C:BRU}.

\end{document}